\documentclass[12pt]{amsart}
\usepackage{amsmath}
\usepackage{amsthm}
\usepackage[pdftex]{graphicx}
\usepackage{amssymb}
\usepackage{mathrsfs}
\usepackage{bbm}
\usepackage{enumerate}
\theoremstyle{plain}  
  \newtheorem{thm}{Theorem}[section]
    \newtheorem*{mainthm}{Main Theorem}
     
  \newtheorem{lem}[thm]{Lemma}    
  \newtheorem{prop}[thm]{Proposition}
\theoremstyle{definition}  
 
\theoremstyle{remark}
  \newtheorem{rem}[thm]{Remark}
  
  \newtheorem*{ack}{Acknowledgements}

\makeatletter
    
    \@addtoreset{equation}{section}
  \makeatother
\newcommand{\C}{\mathbb C} 
\newcommand{\R}{\mathbb R} 
\newcommand{\bbH}{\mathbb H}  
\newcommand{\calC}{\mathcal C} 
\newcommand{\calG}{\mathcal G}  
  
\newcommand{\calQ}{\mathcal Q} 
\newcommand{\sV}{\mathscr V}  
\newcommand{\sE}{\mathscr E}
\newcommand{\sF}{\mathscr F}
\newcommand{\su}{{\mathfrak{su}}} 
 
\newcommand{\alp}{\alpha} 
 
\newcommand{\Si}{\Sigma} 
 
\newcommand{\ep}{\epsilon} 
\newcommand{\vep}{\varepsilon} 
\newcommand{\om}{\omega}
\newcommand{\Del}{\Delta}

\newcommand{\bd}{\partial} 
\newcommand{\wh}{\widehat}
\newcommand{\wt}{\widetilde}

\newcommand{\sm}{\setminus}
\newcommand{\la}{\langle}
\newcommand{\ra}{\rangle}
\newcommand{\one}{\mathbbm{1}}
\newcommand{\mi}{\mathbbm{i}}
\newcommand{\mj}{\mathbbm{j}}
\newcommand{\mk}{\mathbbm{k}}
\newcommand{\vol}{{\mathrm{vol}}}
\newcommand{\matH}{\bbH_{\mathrm{mat}}}
\DeclareMathOperator{\tr}{tr}
\DeclareMathOperator{\Area}{Area}
\DeclareMathOperator{\Span}{Span}
\DeclareMathOperator{\Poly}{Poly}
\DeclareMathOperator{\Gr}{Gr}

\begin{document}
\title{Convex Polyhedra in the $3$-Sphere and Tilings of the $2$-Sphere}
\author{Kentaro Ito}
\address{Graduate School of Mathematics, Nagoya University, Nagoya 464-8602, Japan}
\email{itoken@math.nagoya-u.ac.jp}
\date{\today}

\begin{abstract}
We show that for every convex polyhedral sphere $P$ in $S^3$,  
there exist two canonical, non-edge-to-edge tilings of $S^{2}$ whose tiles are given by 
all the faces of $P$ and the dual convex polyhedral sphere $P^*$ to $P$. 
Under the identifications of $S^{3}$ with the Lie group $SU(2)$, 
and of $S^{2}$ with the unit sphere in the Lie algebra $\su(2)$ of $SU(2)$,  
our result is obtained by considering the set $\wt P$ of outward unit normal vectors to $P$ 
and the maps from $\wt P$ to $S^{2}$ defined by using the left and right Maurer-Cartan forms on $SU(2)$. 
\end{abstract} 

\maketitle

\section{Introduction}

A {\it convex polyhedron} in the unit $3$-sphere $S^{3}$ is a  convex domain bounded by 
a finite number of totally geodesic $2$-spheres.  
A {\it convex polyhedral sphere} is the relative boundary of some convex polyhedron. 
Then, each face of a convex polyhedral sphere 
is isometric to a geodesic polygon in the unit $2$-sphere $S^{2}$. 
The {\it dual} $P^{*}$ to a convex polyhedral sphere $P$ is 
the set of all points in $S^{3}$ whose spherical distance from $P$ is equal to $\pi/2$, 
which is also a convex polyhedral sphere.  
In this paper, we show the following (see also Figure 1): 
\begin{mainthm}
For every convex polyhedral sphere $P$ in $S^3$,  
there exist two canonical, non-edge-to-edge tilings of $S^{2}$ whose tiles are given by 
all the faces of $P$ and the dual convex polyhedral sphere $P^*$ to $P$.  
\end{mainthm}
Here, a polygonal tiling is said to be {\it edge-to-edge} if the intersection 
$T_{1} \cap T_{2}$ of two tiles $T_{1},\,T_{2}$ is a segment, 
it coincides with an edge of $T_{1}$ and an edge of $T_{2}$. 
We say that a polygonal tiling is  {\it non-edge-to-edge} if it is not edge-to-edge. 

We now explain how to construct the tilings in the main theorem. 
Let $P$ be a polygonal sphere that bounds a convex polyhedron $Q$. 
We choose orientations on the faces of $P$ 
determined by outward-pointing normal vectors to the faces of $P=\bd Q$. 
Let $P^{*}$ be the dual convex polyhedral sphere to $P$, 
which bounds a convex polyhedron $Q^{*}$. 
We remark that $Q \cap Q^{*}=\emptyset$. 
We choose orientations on the faces of $P^{*}$ 
determined by {\it inward}-pointing normal vectors to the faces of $P^{*}=\bd Q^{*}$. 
We now develop $P$ into $S^{2}$ as follows: 
Let $\{f_{1}, \ldots, f_{l}\}$ be the set of faces of $P$. 
Each face $f_{i}$ of $P$ is mapped into $S^{2}$ by an orientation preserving isometry 
$\varphi_{i}:f_{i} \to S^{2}$. 
Furthermore, if two faces $f_{i},\,f_{j}$ of $P$ meet along an edge $e$ of $P$ 
with the exterior dihedral angle $\theta$, we require that the following three conditions hold: 
(i) the images $\varphi_{i}(e)$ and $\varphi_{j}(e)$ of edge $e$ lie in  the same geodesic circle 
$\gamma$ on $S^{2}$, 
 (ii) $\varphi_i(f_i)$ and $\varphi_j(f_j)$ lie on the oposite sides of $\gamma$ from each other, and 
 (iii) $\varphi_{j}(f_{j})$ is slid to the right along $\gamma$ 
about distance $\theta$ when viewed from $\varphi_{i}(f_{i})$. 
It turns out that the developed images of the faces of $P$ have mutually disjoint interiors. 
Note that the developed image $\wh P \subset S^{2}$ of $P$ is uniquely determined 
up to orientation preserving isometries on $S^{2}$. 
Similarly, we develop $P^{*}$ into $S^{2}$ via orientation preserving isometries; 
However, in this case, if two faces meet along an edge, we require that 
the image of one face is slid to the {\it left} when viewed from the other.  
We denote by  $\wh P^{*} \subset S^{2}$ the developed image of  $P^{*}$. 
Then, the main theorem claims that 
there is an orientation preserving  isometry  $\psi$ of $S^{2}$ 
such that $\wh P \cup \psi \bigl(\wh P^{*}\bigr)$ gives an non-edge-to-edge tiling of $S^{2}$. 
If we change the direction of the slide in the above construction from right to left and from left to right, 
respectively, we also obtain another tiling of $S^{2}$. 
In this way, we obtain two canonical tilings of $S^{3}$. 

{\it Regular polyhedra} in $S^{3}$ are convex polyhedra whose faces are congruent 
regular spherical polygons.  
As in the case of regular polyhedra in $\R^{3}$, 
there are five types of regular polyhedra in $S^{3}$. 
Each type of a regular polyhedron is parameterized by the length of its edges up to congruence.  
Figure 1 shows two canonical, non-edge-to-edge tilings 
obtained from an icosahedron $P$ in $S^{3}$  and its dual dodecahedron  $P^{*}$. 
The tilings of $S^{2}$ in the main theorem obtained from regular polyhedra in $S^{3}$ 
and their dual polyhedra are already known in \cite{Adams} as kaleidoscope tilings. 
However, tilings obtained from general convex polyhedra and their dual polyhedra 
seem to be new as far as we know. 

\begin{figure}[h]
\begin{center}
\includegraphics[height=3.5cm]{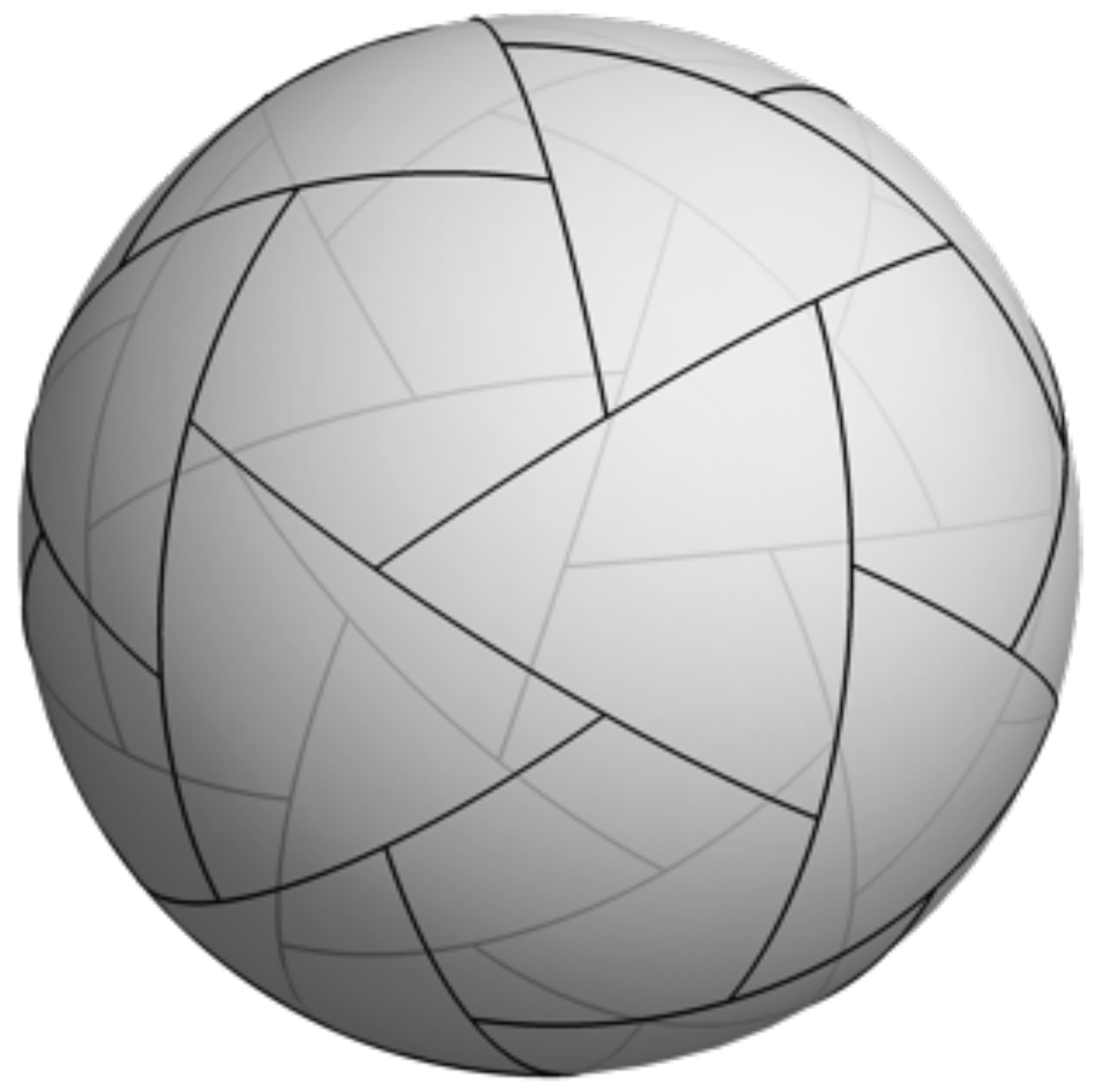}
\hspace{3cm}
\includegraphics[height=3.5cm]{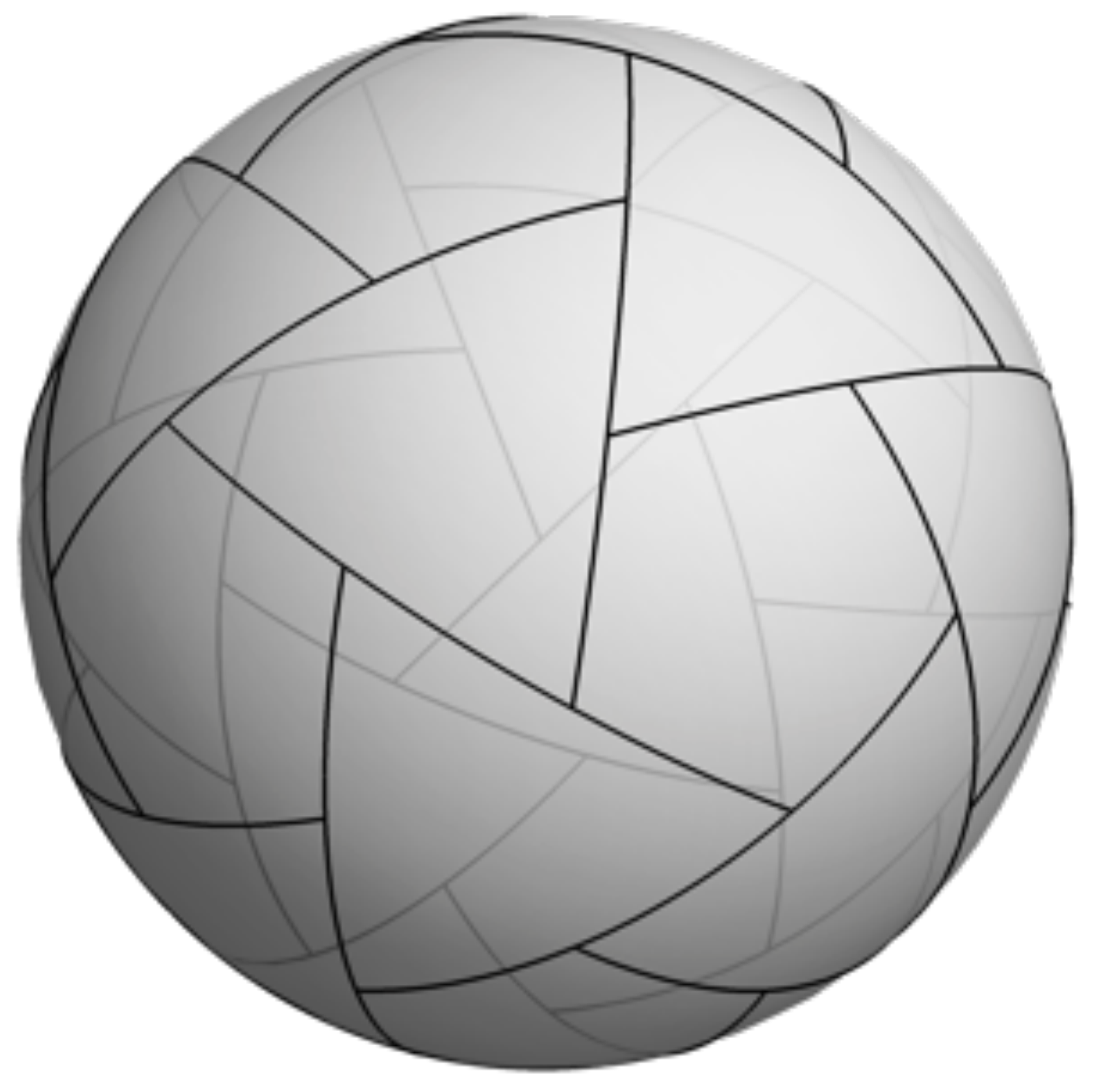}
\caption{Two canonical non-edge-to-edge tilings of $S^{2}$ 
obtained from a regular icosahedron in $S^{3}$ 
and its dual dodecahedron. }
\end{center}
\end{figure}

We now discuss the proof of the main theorem. 
Although the construction of tilings described above works well locally (see Lemma 4.2), 
it is difficult to see that tilings of the whole $S^{2}$ are also obtained.   
Thus, we proceed in another way. 
For a given polyhedral sphere $P$, we consider the subset $\wt P$ 
of the unit tangent bundle $T^{1}S^{3}$ of $S^{3}$ that consists of the outward unit normal vectors to $P$. 
Note that the subset of vectors in $\wt P$ based at a vertex $v$ of $P$ 
is naturally identified with the face $v^{*}$ of $P^{*}$ that is dual to $v$. 
Similarly, the subset of vectors in $\wt P$ whose base points lie on  
an edge $e$ of $P$ is identified with 
the product $e \times e^{*}$, where $e^{*}$ is an edge of $P^{*}$ that is dual to $e$. 
From this observation, we obtain an edge-to-edge tiling of $\wt P$ 
whose tiles are all faces of $P$ and $P^{*}$, and planer rectangles 
corresponding to the edges of $P$. 
 We show that $\wt P$ is homeomorphic to $S^{2}$ (see Proposition 5.1). 

To prove the main theorem, 
we identify $S^{3}$ with the Lie group $SU(2)$ and 
$S^{2}$ with the unit sphere in the Lie algebra $\su(2) \cong \R^{3}$ of $SU(2)$.  
Then, the left and right Maurer-Cartan forms on $SU(2) \cong S^{3}$ 
induce maps $\om,\,\om':T^{1}S^{3} \to S^{2}$. 
Let us consider the restriction $\om|_{\wt P}$ (resp. $\om'|_{\wt P}$) of $\om$ 
(resp. $\om'$) to $\wt P$, which is also denoted by $\om$ (resp. $\om'$).  
We show the following in Section 7: 
(i) the map $\om:\wt P \to S^{2}$ is surjective and takes each of the faces of $P$ or $P^{*}$ 
(which is regarded as a subset of $\wt P$) isometrically into $S^{2}$, 
(ii) the images of planar rectangles are geodesic segments in $S^{2}$, 
and (iii) the images of the faces of $P$ and $P^{*}$ have mutually disjoint interiors. 
The map $\om':\wt P \to S^{2}$ also has the same property.  
Thus, we obtain two non-edge-to-edge tilings of $S^{2}$ corresponding to $\om$ and $\om'$. 
We remark that  these tilings are the same as those obtained by developing 
$P$ and $P^{*}$ into $S^{2}$ as described above. 

We now explain where the idea of this paper comes from. 
As the $3$-sphere $S^{3}$ is identified with $SU(2)$,  $3$-dimensional anti-de Sitter space 
$AdS^{3}$ is identified with $SU(1,1)$. 
Therefore, $S^{3}$ and $AdS^{3}$ have many common features. 
It was observed by Mess \cite{Mess} that each spacelike, complete  
convex polyhedral plane in $AdS^{3}$ without vertices provides two tilings of the hyperbolic plane $H^{2}$. 
This observation was used in \cite{Mess} to provide an alternative proof of Thurston's earthquake theorem. 
Our result can be seen as the analog of Mess's observation in the setting of 
convex polyhedral spheres in $S^{3}$, which inevitably have vertices. 
Similarly, each spacelike, complete convex polyhedral plane in $AdS^{3}$ with vertices 
 and its dual plane provide tilings of $H^{2}$. 
The paper on this topic is now in preparation. 

This paper is organized as follows: 
In Section 2, we give the basic notion and definitions concerning the unit 3-sphere $S^{3}$. 
In Sections 3 and 4, we consider convex polyhedra in $S^{3}$ and their dual polyhedra   
and provide some basic properties. 
In Section 5, we consider the set $\wt P$ of outward unit normal vectors 
to a convex polyhedral sphere $P$ and show that $\wt P$ is homeomorphic to $S^{2}$. 
In Section 6, we explain the identification of $S^{3}$ with $SU(2)$ and introduce the 
left and right Maurer-Cartan forms. 
In Section 7, we give the proof of the main theorem by 
combining all the results from the previous chapters. 

\begin{ack}
The author wishes to express his thanks to Shin Nayatani and Hiroki Fujino for their 
interest and many helpful suggestions. 
\end{ack}

\section{The unit $3$-sphere $S^{3}$}

This section provides definitions of the spaces with which we are concerned. 
We equip the Euclidean $4$-space $\R^{4}$ with a standard inner product 
$$
\la x,y\ra=x_{1}y_{1}+x_{2}y_{2}+x_{3}y_{3}+x_{4}y_{4}. 
$$
The $3$-dimensional unit sphere 
$$
S^{3}=\{x \in \R^{4} : \la x,x\ra=1\}
$$ 
is a Riemannian manifold with constant sectional curvature $1$. 
Denoting the distance between two points $x,y$ in $S^{3}$ by $d(x,y)$,  we have  
$\la x,y\ra=\cos d(x,y)$. 
The tangent space $T_{x}S^{3}$ of $S^{3}$ at $x\in S^3$ 
is identified with the  orthogonal subspace:  
$$
T_{x}S^{3} \cong x^{\perp}:=\{\nu \in \R^{4} : \la x,\nu\ra=0\}. 
$$ 
Therefore, the tangent bundle $TS^{3}$ and the unit tangent bundle 
$T^1S^3$ of $S^{3}$ are realized as submanifolds of  $\R^4 \times \R^4$ as follows: 
\begin{align*}
TS^{3}&=\{(x,\nu)\in \R^{4} \times \R^{4} : \la x,x\ra=1,\,\la x,\nu\ra=0\}, \\
T^{1}S^{3}&=\{(x,\nu)\in \R^{4} \times \R^{4} : \la x,x\ra=\la \nu,\nu\ra=1,\,\la x,\nu\ra=0\}. 
\end{align*}

Given a point $\nu \in S^{3}$, we define the {\it hyperplane} $\Pi_{\nu}$ of $S^{3}$ 
associated with $\nu$ by 
$$
\Pi_{\nu}:=\nu^{\perp} \cap S^{3}=\{x \in S^{3} : \la x,\nu\ra=0\}=\{x \in S^{3} : d(x ,\nu)=\pi/2\},  
$$
which is a  totally geodesic $2$-sphere in $S^{3}$.  
Note that the tangent space $T_{x} \Pi_{\nu}$ of $\Pi_{\nu}$ at $x \in \Pi_{\nu}$ 
is identified with the subspace $x^{\perp} \cap \nu^{\perp}$ of $\R^{4}$.  
We next define the closed {\it half-sphere} $\Delta_{\nu}$ associated with $\nu \in S^{3}$ by 
$$
\Delta_{\nu}:=\{x \in S^{3} : \la x,\nu\ra \le 0\}=\{x \in S^{3} :  d(x,\nu) \ge \pi/2\}. 
$$
Thus, the relative boundary of $\Delta_{\nu}$ is $\Pi_{\nu}$. 
The complement $S^{3} \setminus \Delta_{\nu}$ is denoted by $\Delta_{\nu}^{c}$, 
which is the open ball of radius $\pi/2$ centered at $\nu$.

Throughout this paper, careful treatments of the orientations of the faces of polyhedra are needed. 
We regard an {\it orientation} on a manifold $M$ as an equivalence class of ordered 
bases for the tangent space $T_{x}M$ at each $x \in M$.   
We define the orientation on $S^{3}$  as follows: 
An ordered basis $(u_{1},u_{2},u_{3})$ for the tangent space $T_{x}S^{3}$ at $x \in S^{3}$ is 
in the positive orientation on $S^{3}$ if the ordered basis $(x, u_{1},u_{2},u_{3})$ for 
$T_{x}\R^{4} \cong \R^{4}$ is in the positive orientation on $\R^{4}$. 
Let $S$ be an orientable, smoothly embedded surface in $S^{3}$. 
Choosing an orientation on $S$ is equivalent to choosing a unit normal vector field $n$ on $S$: 
If  the ordered basis $(u_{1},u_{2})$ for $T_{p} S$ ($p \in S$) is in the positive orientaiton on $S$, 
we choose the unit normal vector $n(p)$ to $S$ so that the ordered basis 
$(u_{1},u_{2},n(p))$ for $T_{p}S^{3}$ is in the positive orientation on $S^{3}$.  
In this situation, we say that  the unit normal vector field $n$ determines the orientation on $S$. 
Throughout this paper, we  adopt the orientation on a hyperplane 
$\Pi_{\nu}$ for $\nu \in S^{3}$ which is determined by the unit normal vector field $\nu$ on $\Pi_{\nu}$. 

A {\it geodesic} $\gamma$ on $S^{3}$ is the intersection of $S^{3}$ 
with some 2-dimensional subspace $V \subset \R^{4}$. 
Choosing an orientation of  $V$ is equivalent to choosing an orientation of  $\gamma=S^{3} \cap V$: 
If $(u_{1},u_{2})$ is in the positive orientation of $V$, 
 we orient $\gamma$ so that the parametrization $\gamma(t)=(\cos t)u_{1}+(\sin t)u_{2}$ 
is the positive orientation on $\gamma$.  

The {\it dual geodesic} $\gamma^{*}$ to a geodesic $\gamma=S^{3} \cap V$ is defined to be 
$\gamma^{*}=S^{3} \cap V^{\perp}$, where $V^{\perp} \subset \R^{4}$ 
is the orthogonal complement of $V$. 
If $\gamma$ is oriented, we give an associated orientation to $\gamma^{*}$ as follows: 
If $(u_{1},u_{2})$ is in the positive orientation of $V$, we choose an orientation 
$(u_{3},u_{4})$ of $V^{\perp}$ 
so that $(u_{1},u_{2},u_{3},u_{4})$ is in the positive orientation of $\R^{4}$. 

\section{Convex polyhedra in $S^{3}$ and their dual polyhedra}

In this section, we consider convex polyhedra in $S^{3}$ and their dual polyhedra.   
As we describe in the proof of Lemma 3.1, convex polyhedra in $S^{3}$ naturally 
correspond to convex polyhedra in the Euclidean space $\R^{3}$. 
Hence, the properties of convex polyhedra in $S^{3}$ 
and their dual polyhedra are deduced from those in $\R^{3}$. 
We refer the reader to Gr\"{u}nbarm \cite{Gru} and Matou\v{s}ek \cite{Mat} for information on 
convex polyhedra in the Euclidean spaces and their dual polyhedra. 
Some properties of convex polyhedra in $S^{3}$ and their dual polyhedra can also be found in \cite{HR}. 

We first fix some terminology for a subset $X$ of $S^{3}$. 
We say that  $X$ is {\it convex} if for any $x,y \in X$, 
(one of) the shortest geodesic arc on $S^{3}$ connecting $x$ and $y$ is contained in $X$. 
We say that $X$ is {\it hemispherical} if 
$X$ is contained in some open half-sphere; more precisely,  
there exists $x_{0} \in S^{3}$ such that $X \subset \Delta_{x_{0}}^{c}$, 
where $\Delta_{x_{0}}^{c}=S^{3} \setminus \Delta_{x_{0}}$ is the open ball 
of radius $\pi/2$ centered at $x_{0}$.  
We say that $X$ is {\it non-planar} if there is no hyperplane of $S^{3}$ containing $X$. 

We now consider convex polyhedra in $S^{3}$.  
Take a hemispherical, non-planar finite subset $\sV$ of $S^{3}$.  
The {\it cone} $\calC(\sV) \subset \R^{4}$ spanned by $\sV$ is defined by
$$
\calC(\sV):=\left\{\sum_{v \in \sV} a_{v}v  \in \R^{4}: a_{v} \ge 0\right\}. 
$$
The set $CH(\sV):=\calC(\sV) \cap S^{3}$ is called  the {\it convex hull} of $\sV$ in $S^{3}$.  
Since $\sV$ is non-planar, $CH(\sV) \subset S^{3}$ has an interior point.  
We say that $CH(\sV)$ is the  {\it convex polyhedron} generated by $\sV$. 
The relative boundary $P=\bd Q$  of $Q=CH(\sV)$  is homeomorphic to $S^{2}$ and 
is called a {\it convex polyhedral sphere} generated by $\sV$.   
In what follows, we assume that each element $v \in \sV$ is {\it extremal} in $CH(\sV)$; 
that is, each $v \in \sV$ satisfies the condition $v \notin CH(\sV \sm \{v\})$. 
Then, each $v \in \sV$ is called a {\it vertex} of $P$,  and we also denote $\sV$ by $\sV(P)$. 
The edges and faces of $P$ are also naturally defined. 
We denote the sets of edges and faces of $P$ by $\sE(P)$ and  $\sF(P)$, respectively.  

We next define the dual polyhedron $Q^{*}$ to the convex  polyhedron $Q$: 
The {\it dual cone}  $\calC(\sV)^*$ to the cone $\calC(\sV)$ is defined by 
$$
\calC(\sV)^{*}:=\{x \in \R^{4}: \la x, y\ra \le 0 \ (\forall y \in \calC(\sV))\}. 
$$
Then, $Q^{*}=\calC(\sV)^{*} \cap S^{3}$ is called the {\it dual polyhedron} 
to the convex polyhedron $Q$, and $P^{*}=\bd (Q^{*})$  
the {\it dual polyhedral sphere} to $P=\bd Q$. 
We define subsets $f^{*},e^{*}$ and $v^{*}$ of $P^{*}$ for elements 
$f \in \sF(P)$, $e \in \sE(P)$ and $v \in \sV(P)$, respectively, as follows: 
\begin{align*}
f^{*}&:=\{x \in P^{*}: \la  x,y\ra=0 \ (\forall y \in f)\},  \\
e^{*}&:=\{x \in P^{*}: \la  x,y\ra=0 \ (\forall y \in e)\},  \\
v^{*}&:=\{x \in P^{*}: \la  x,v \ra=0\}.   
\end{align*}
\begin{lem}[\cite{Gru}, \cite{Mat}]
The dual polyhedron $Q^{*}$ to the convex polyhedron $Q=CH(\sV)$ is actually a convex polyhedron. 
Furthermore, the sets $\{f^{*} :  f \in \sF(P)\}$,  $\{e^{*} : e \in \sE(P)\}$  and $\{v^{*} : v \in \sV(P)\}$ 
defined above coincide with the set of vertices, edges and faces of 
the dual polyhedral sphere $P^{*}$ to $P$, respectively. 
In addition, we have $(Q^{*})^{*}=Q$. 
\end{lem}

\begin{proof}
We identify $\R^{3}$ with the hyperplanes 
$H_{\pm}=\{(x_{1},x_{2},x_{3},x_{4}) \in \R^{4} : x_{4}=\pm 1\}$ in $\R^{4}$. 
Since we are assuming that $\sV$ is hemispherical and non-planar, we may assume that 
$N=(0,0,0,1) \in \R^{4}$ is an interior point of the convex polyhedron 
$Q=CH(\sV) \subset S^{3}$ and that $Q \subset \Delta_{N}^{c}$. 
Then $\calQ=\calC(\sV) \cap H_{+}$ is a bounded convex polyhedron in $\R^{3} \cong H_{+}$ 
with interior point $(0,0,0) \in \R^{3}$.  
Note that $\calQ$ corresponds to $Q \subset S^{3}$ via the radial projection. 
We also define a subset $\calQ^{*} \subset \R^{3} \cong H_{-}$ by $\calQ^{*}=\calC(\sV)^{*} \cap H_{-}$, 
which corresponds to $Q^{*} \subset S^{3}$.  
Then it can be seen that 
$$
\calQ^{*}=\{x \in \R^{3}: \la x, q\ra \le 1 \, (\forall q \in \calQ)\}
$$
holds, which is the definition adopted in \cite{Gru} and \cite{Mat} 
of the dual polyhedron $\calQ^{*}$ to a convex polyhedron $\calQ$ in $\R^{3}$. 
Then all the properties in Lemma 3.1 are obtained from the arguments 
in \cite[Section 3.4]{Gru} or \cite[Chapetr 5]{Mat}. 
\end{proof} 

Throughout this paper, for a polyhedral sphere $P$, 
we fix the orientation on each face of $P$ (resp. $P^{*}$) induced from 
the outward (resp. inward) unit normal vectors to faces of $P$ (resp. $P^{*}$). 
In other words, the orientation on a face $f$ of $P$ 
coincides with that of $\Pi_{f^{*}}$, 
and the orientation on a dual face $v^{*} \subset P^{*}$ of a vertex $v \in P$ 
coincides with that of $\Pi_{-v^{*}}$.  
 
\section{Some properties of convex polyhedral spheres}

In this section, we introduce some relations between convex polyhedral spheres and their 
dual polyhedral spheres, which are used in what follows. 

Let $P$ be a convex polyhedral sphere in $S^{3}$. 
Suppose that two faces $f_1$ and $f_2$ of $P$ meet along an edge $e$. 
Then the {\it exterior dihedral angle} of $f_{1}, f_{2}$ at $e$ is defined to be 
the spherical distance  $d(f_{1}^{*},f_{2}^{*})$ between the dual vertices 
$f_{1}^{*}$ and $f_{2}^{*}$ of $P^{*}$ to $f_{1}$ and $f_{2}$, respectively, 
or the length of the edge $e^{*}$ of $P^{*}$. 

We need the following two fundamental lemmas, and the proof for the first one is left for the reader. 
\begin{lem}
Let $P$ be a convex polyhedral sphere in $S^{3}$. 
Let $e$ be an edge of $P$ and choose an orientation of $e$ arbitrarily.    
Let $v_{1}$  (resp. $v_{2}$) be the starting (resp. terminal) point of $e$.  
Let $f_{1},f_{2}$ be the faces of $P$ that meet along $e$, and suppose that 
$e$ and $\bd f_{1}$ (resp. $\bd f_{2}$) have the  same  (resp. opposite) orientations. 
In other words, $f_{1}$ is on the left side of $e$, whereas $f_{2}$ is on the right side.
Let $e^{*}$ be the dual edge to $e$ equipped with the orientation 
associated with that of $e$ (see Section 2).  
Then, $f_{1}^{*}$ (resp. $f_{2}^{*}$) is the starting (resp. terminal) point of $e^{*}$, 
and $v_{1}^{*}$ is on the right side of $e^{*}$, whereas $v_{2}^{*}$ is on the left side 
(see Figure 2). 
\end{lem}

\begin{figure}[h]
\begin{center}
\includegraphics[height=4.5cm]{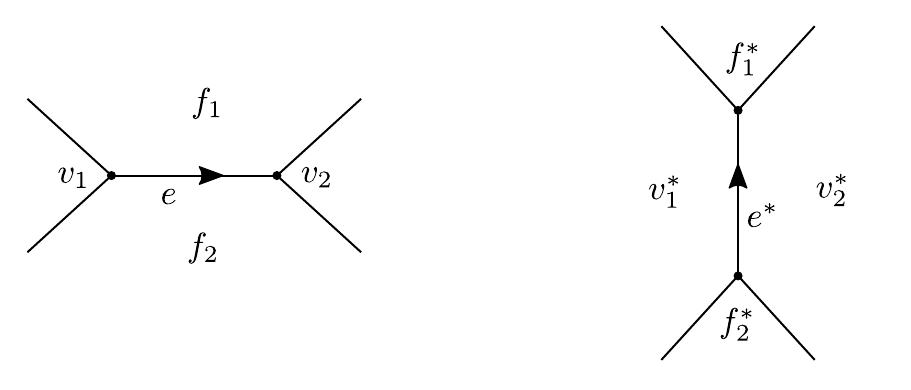}
\caption{Configurations of vertices, edges and faces of $P$ (left) and their dual entities in $P^{*}$ (right). }
\end{center}
\end{figure}

\begin{lem}[Proposition 2.7 in \cite{HR}]
Let $P$ be a convex polyhedral sphere in $S^{3}$. 
If the angle of face $f$ of $P$ at vertex $v$ is $\theta$, then 
the angle of face $v^{*}$ of $P^{*}$ at vertex $f^{*}$ is equal to $\pi-\theta$. 
\end{lem}

Using Lemma 4.2, we have the following: 
\begin{lem}
Let $P$ be a convex polyhedral sphere and $P^{*}$ its dual polyhedral sphere.  
Then, we have 
$\Area(P)+\Area(P^{*})=4 \pi$. 
\end{lem}
\begin{proof}
Suppose that $P$ has $l$ faces,  $m$ edges and $n$ vertices, 
and let $\sF(P)=\{f_{1}, \ldots, f_{l}\}$ be the set of faces of $P$. 
Since $P$ is homeomorphic to the $2$-sphere, we have $l-m+n=2$. 
Suppose that face $f_{i}$ is an $a_{i}$-gon ($i=1, \ldots, l$). 
Then, $2m=\sum_{i=1}^{l}a_{i}$ holds. 
Let $\{v_{ij} : 1 \le j \le a_{i}\}$ be the set of vertices of $f_{i}$, and let
$\theta_{ij}$ be the angle of $f_{i}$ at $v_{ij}$. 
The Gauss-Bonnet theorem then implies that 
\begin{align}
\Area(f_{i})=2\pi-\sum_{j=1}^{a_{i}} (\pi-{\theta_{ij}}). 
\end{align}
From Lemma 4.2, it can be seen that  $\sum_{j=1}^{a_{i}} (\pi-{\theta_{ij}})$ is the sum of 
all angles in $P^{*}$ that meet at vertex $f_{i}^{*}$, i.e., the cone angle of $P^{*}$ at $f_{i}^{*}$. 
Therefore, by summing (4.1) over $i$, we obtain 
\begin{align*}
\Area(P)=2l\pi -\Theta(P^{*}), 
\end{align*}
where $\Theta(P^{*})$ is the sum of all angles of all faces of $P^{*}$. 
By duality, we have 
\begin{align}
\Area(P^{*})=2n\pi -\Theta(P), 
\end{align}
where $\Theta(P)$ is the sum of all angles of all faces of $P$. 

On the other hand, by summing (4.1) over $i$ again, we obtain 
\begin{align}
\Area(P)&=2l\pi-\sum_{i=1}^{l}{a_{i}}\pi+\sum_{i=1}^{l}\sum_{j=1}^{a_{i}} {\theta_{ij}}.  
\end{align}
 Since $\sum_{i=1}^{l}a_{i}=2m$, it follows from (4.3) that 
 \begin{align}
\Area(P)=2(l-m)\pi+\Theta(P). 
\end{align}
By summing (4.2) and (4.4) and using $l-m+n=2$,  we obtain $\Area(P)+\Area(P^{*})=4 \pi$. 
\end{proof}

\section{Outward unit normal vectors}

Let $P$ be a convex polyhedral sphere. 
In this section, we consider the subset $\wt P$ of the unit tangent bundle $T^{1} S^{3}$ of $S^{3}$  
that consists of outward unit normal vectors to $P$ and 
equip $\wt P$ a piecewise smooth Riemannian metric without cone-like singular points. 

Let $P$ be a convex polyhedral sphere in $S^{3}$. 
A hyperplane $\Pi_{\nu} \,(\nu \in S^{3})$ is said to be 
a {\it support plane} for $P$  at $x \in P$ if both 
$P \subset \Del_{\nu}$ and $x \in \Pi_{\nu} \cap P$ are satisfied. 
Note that $\Pi_{\nu}$ is a support plane for $P$ at $x \in P$ if and only if 
$\nu \in P^{*}$ and $d(x,\nu)=\pi/2$.  
When $\Pi_{\nu}$ is a support plane for $P$ at $x$, 
by abuse of language, we say that 
$\nu$ is an {\it outward unit normal vector} to $P$ at $x$; 
although, if $x$ lies on an edge or a vertex of $P$, it is not uniquely determined.  
Let us denote by $\wt P$ the set of outward unit normal vectors to $P$: 
\begin{align*}
\wt P
&:=\{(x,\nu) \in T^{1} S^{3} : 
x \in P, \, \nu \text{ is an outward unit normal vector to } P \text{ at }x \}\\
&=\{(x,\nu) \in T^{1}S^{3} : x \in P,\,\nu \in P^{*},\, d(x,\nu)=\pi/2\}. 
\end{align*}
We equip $\wt P$ with the topology induced from that of $T^{1}S^{3}$. 
Later, we prove later that $\wt P$ is homeomorphic to $S^{2}$. 

Let $\sF(P)$, $\sE(P)$ and $\sV(P)$ be the sets of faces, edges and vertices of $P$,  respectively. 
For elements $f \in \sF(P)$, $e \in \sE(P)$ and  $v \in \sV(P)$, 
we define subsets $\tilde f$, $\tilde e$ and $\tilde v$ of $\wt P$, respectively, as follows: 
\begin{align*}
\tilde f&:=f \times \{f^{*}\}=\{(x,f^{*}) \in T^{1}S^{3}: x \in f\},  \\
\tilde e&:=e \times e^{*}=\{(x,\nu) \in T^{1}S^{3} : x \in e,\,\nu \in e^{*}\}, \\
\tilde v&:=\{v\} \times v^{*}=\{(v,\nu) \in T^{1}S^{3} : \nu \in v^{*}\}. 
\end{align*}
It is easily seen that these subsets cover the whole  $\wt P$: 
\begin{align*}
\wt P=\left(\bigcup_{f \in \sF(P)} \tilde f \right) \cup 
\left(\bigcup_{e \in \sE(P)}\tilde e \right) \cup 
\left(\bigcup_{v \in \sV(P)} \tilde v \right). 
\end{align*}
We now equip a piecewise smooth, continuous Riemannian metric on $\wt P$. 
We put an orientation and a Riemannian metric on $\tilde f$ (resp. $\tilde v$) so that the canonical 
homeomorphism $f \to \tilde f; \, x \mapsto (x,f^{*})$ 
(resp. $v^{*} \to \tilde v; \, \nu \mapsto (v,\nu)$) is an orientation preserving isometry.  
We next put an orientation and a Riemannian metric on $\tilde e$. 
Let us first choose an orientation of $e$ arbitrarily.  
Then, the orientation of the dual-edge $e^{*}$ is canonically determined, 
as explained in Section 2. 
Let $e(s)\ (0 \le s \le a)$ and $e^{*}(t)\ (0 \le t \le b)$ 
be the  arc length parametrizations of $e$ and  $e^{*}$ that  are 
consistent with the orientations on $e$ and $e^{*}$, respectively.  
Then, we put an orientation and Riemannian metric on $\tilde e$ so that the canonical 
homeomorphism 
$$
\R^{2} \supset [0,a] \times [0,b] \ni (s,t) \mapsto (e(s),e^{*}(t)) \in \tilde e
$$
is  an orientation-{\it reversing} isometry, 
where $\R^{2}$ is equipped with the canonical Euclidean metric. 
Thus, we obtain spherical polygons $\{\tilde f : f \in \sF(P)\}$ and 
$\{\tilde v : v \in \sV(P)\}$, and planner rectangles $\{\tilde e : e \in \sE(P)\}$ in $\wt P$. 
We denote by  $\Poly(\wt P)$ the set of these polygons in $\wt P$. 
Then, we have the following: 
\begin{prop}
Let $P$ be a convex polyhedral sphere in $S^{3}$. 
The set $\wt P$ of the outward unit normal vectors to $P$ is 
tiled by all elements in $\Poly(\wt P)$.  This tiling is edge-to-edge and 
the orientations on the polygons in $\Poly(\wt P)$ fit together to give an orientation on $\wt P$. 
In addition, $\wt P$ is homeomorphic to $S^{2}$. 
\end{prop}

\begin{proof}
As observed above, $\wt P$ is  covered by polygons in $\Poly(\wt P)$; i.e., 
$\wt P=\bigcup_{\wp \in \Poly(P)}\wp$.  
To prove the proposition, we examine the intersection of two polygons in  $\Poly(\wt P)$ 
in the following six cases (i)--(vi): 
\begin{enumerate}[(i)]
\item For any two distinct elements $f_{1} ,f_{2} \in \sF(P)$, $\tilde f_{1} \cap \tilde f_{2} =\emptyset$. 
\item For any two distinct elements $v_{1} ,v_{2} \in \sV(P)$, $\tilde v_{1} \cap \tilde v_{2} =\emptyset$. 
\item For $f \in \sF(P)$ and $e \in \sE(P)$, 
$\tilde f \cap \tilde e \ne \emptyset$ if and only if  $e$ is an edge of $f$. 
In this case, $\tilde f \cap \tilde e$ is equal to the set 
$e \times \{f^{*}\}=\{(x,f^{*}): x \in e\}$, 
which is an edge of $\tilde f$ and $\tilde e$. 
\item For $e \in \sE(P)$ and $v \in \sV(P)$, 
$\tilde e  \cap \tilde v\ne \emptyset$ if and ponly if 
$v$ is an end point of $e$. 
In this case, $\tilde e  \cap \tilde v$ is equal to the set 
$\{v\} \times e^{*}=\{(v,\nu): \nu \in e^{*}\}$,  
which is an edge of $\tilde e$ and $\tilde v$. 
\item For $f \in \sF(P)$ and $v \in \sV(P)$, 
$\tilde f \cap \tilde v \ne \emptyset$ if and only if $v$ is  a vertex of $f$. 
In this case, $\tilde f \cap \tilde v$ is the point $(v, f^{*}) \in \wt P$, 
which is a vertex of $\tilde f$ and $\tilde v$. 
\item For any two distinct elements $e_{1}, e_{2} \in \sE(P)$, 
$\tilde e_{1} \cap \tilde e_{2} \ne \emptyset$ if and only if both $e_{1}$ and $e_{2}$ are 
edges of the same face $f \in \sF(P)$ and have a common end point $v \in \sV(P)$. 
In this case, $\tilde e_{1} \cap \tilde e_{2}$ is the point $(v, f^{*}) \in \wt P$, 
which is a vertex of $\tilde e_{1}$ and $\tilde e_{2}$. 
\end{enumerate}
From the above observations (i)--(vi), it can bee seen that polygons in $\Poly(\wt P)$ 
have mutually disjoint interiors. 
It can also be seen from (iii) and (iv) that for any edge $\vep$ of any polygon $\wp$ in $\Poly(\wt P)$, 
there exists exactly one polygon $\wp' ( \ne \wp)$ in $\Poly(\wt P)$  
such that $\vep$ is also an edge of $\wp'$ and that $\wp$ and $\wp'$ are glued along $\vep$. 
In addition, the orientations on $\wp$ and $\wp'$ fit together. 
We explain this only for the case where $\wp=\tilde f$ and $\wp'=\tilde e$ for some 
$f \in \sF(P)$ and its edge $e$ (see case (iii)). 
In this case,  $\tilde f$ and $\tilde e$ are glued along $e \times \{f^{*}\}$ in an edge-to-edge manner.   
We need to show that if we equip $\bd \tilde f$ and $\bd \tilde e$ with the orientations 
induced from that of $\tilde f$ and $\tilde e$, respectively, 
the orientations restricted to $e \times \{f^{*}\} \subset \bd \tilde f$ and 
 $e \times \{f^{*}\} \subset \bd \tilde e$ are the opposite. 
 This can be seen as follows:  
 Let us choose an orientation of $e$ arbitrary,  and equip $e \times \{f^{*}\}$ the induced orientation. 
 When the orientation of $e$ agrees with that of $\bd f$, 
 by Lemma 4.1, $f^{*}$ is the starting point of 
 $e^{*}$, and hence, the orientation of $e \times \{f^{*}\}$ is opposite to that of $\bd \tilde e$. 
 The argument for the case when the orientation of $e$ is opposite to that  of $f$ is similar. 

We next observe the configurations of the polygons in $\Poly(\wt P)$ at their vertices. 
Note that any vertex of any polygon in $\Poly(\wt P)$ is of the form 
$(v,f^{*})$ for some $f \in \sF(P)$ and some vertex $v$ of $f$. 
Let $e_{1}$ and $e_{2}$ be edges of $f$ that have a common end point $v$.  
Then, the polygons adjacent to point $(v,f^{*})$ in $\wt P$ are exactly $4$-polygons 
$\tilde f, \tilde e_{1}, \tilde v$ and $\tilde e_{2}$. 
Furthermore, by observing how these polygons are glued along their edges, 
it can be seen that  these polygons come together to make a neighborhood of the point $(v,f^{*})$ 
(see Figure 3). 
Moreover, it follows from Lemma 4.2 that the angles of polygons 
$\tilde f, \tilde e_{1}, \tilde v$ and $\tilde e_{2}$ at vertex $(v,f^{*})$ add up to $2\pi$. 

\begin{figure}[h]
\begin{center}
\includegraphics[height=4cm]{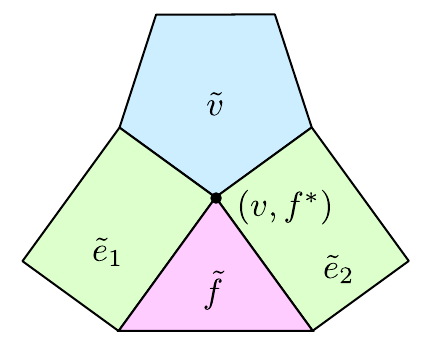}
\caption{Polygons in $\Poly(\wt P)$ that  are adjacent to point $(v,f^{*})$. }
\end{center}
\end{figure}

From the above arguments, it can be concluded that $\wt P$ is homeomorphic to an 
orientable $2$-dimensional topological manifold without boundaries. 
Moreover, $\wt P$ has a piecewise smooth, continuous metric without cone-like singular points.  

Finally, we show that $\wt P$ is homeomorphic to $S^{2}$. 
Since $\wt P$ is tiled by geodesic polygons and has no cone-like singular point, 
we can apply the Gauss-Bonnet theorem to $\wt P$ to obtain 
\begin{align}
\int_{\wt P}K \,\vol=2 \pi\, \chi(\wt P), 
\end{align}
where $K$ is the Gauss curvature defined on the interiors of the polygons in $\Poly(\wt P)$, 
and $\chi(\wt P)$ is the Euler number of $\wt P$. 
Since $K \equiv 1$ on $\tilde f$ for $f \in \sF(P)$ and on $\tilde v$ for $v \in \sV(P)$,
 and $K \equiv 0$ on $\tilde e$ for $e \in \sE(P)$,  
the left-hand side of (5.1) is equal to $\Area(P)+\Area(P^{*})$, 
which in turn is equal to $4\pi$ by Lemma 4.3. 
Therefore we have $\chi(\wt P)=2$,  which implies that $\wt P$ is homeomorphic to $S^{2}$. 
\end{proof}

We remark that the metric $\wt P$ coincides with the induced metric 
from the Sasaki metric on $TS^{3}$.    

\section{Identification of $S^{3}$ with $SU(2)$. }

In the proof of the main theorem, 
we make use of the Lie group structure on $S^{3}$.  
Therefore, in this section, we explain the identification of $S^{3}$ with the Lie group $SU(2)$. 
In addition, we introduce the left and right Maurer-Cartan forms on $SU(2)$.  

We set elements 
$\one, \mi,\mj,\mk$ in the algebra $M(2,{\C})$ of $2 \times 2$ complex matrices as 
\begin{align*}
\one=\left(\begin{array}{cc}
 1 & 0  \\
 0 &  1
\end{array}\right), \quad
\mi=\left(\begin{array}{cc}
 i & 0  \\
 0 &  -i
\end{array}\right), \quad
\mj=\left(\begin{array}{cc}
 0 & 1  \\
 -1 &  0
\end{array}\right), \quad
\mk=\left(\begin{array}{cc}
 0 & i  \\
 i &  0
\end{array}\right). 
\end{align*}
We then obtain a matrix model $\matH$ for the algebra of quaternions $\bbH$ as 
\begin{align*}
\bbH_{\mathrm{mat}}
=\Span_{\R}(\one, \mi,\mj,\mk)
=\left\{\left(\begin{array}{cc}
 \alp & \beta  \\
 -\bar \beta & \bar \alp 
\end{array}\right) : \alp,\beta \in \C\right\}. 
\end{align*}
We equip $\matH$ with an inner product 
\begin{align}
\la x,y\ra=\frac{1}{2}\tr( x y^{*}) \quad (x,y \in \matH), 
\end{align}
where $y^{*}$ is the Hermitian conjugate of $y$. 
In particular, we have  $\la x,x\ra=\det x$ for $x \in \matH$.  
Then, the map $\R^{4} \to \matH$ defined by 
\begin{align*}
\R^{4} \ni (x_{1},x_{2},x_{3},x_{4}) 
\longmapsto x_{1} \one+x_{2}\mi+x_{3}\mj+x_{4}\mk \in \matH
\end{align*}
 is bijective and isometric.  
 We identify $\matH$ with  $\R^{4}$  by this map. 
Then, the group 
\begin{align*}
SU(2)
&=\{x \in M(2,\C): xx^{*}=\one, \, \det x=1\}\\
&=\{x \in \matH : \la x,x\ra=1\}
\end{align*}
is identified with $S^{3} \subset \R^{4}$, 
and the tangent space $T_{\one} S^{3}$ of $S^{3}=SU(2)$ at $\one$ is identified with 
the Lie algebra $\su(2)$ of $SU(2)$: 
\begin{align*}
\su(2)=\{x \in M(2,\C) : x+x^{*}=O,\, \tr x=0\}=\Span_\R(\mi,\,\mj,\,\mk)=\one^{\perp}. 
\end{align*}
We remark that the restriction of the inner product (6.1) to $\su(2)$ coincides with 
the Killing form on $\su(2)$ up to some negative multiplicative constant. 
Let us denote by $\Si$ the unit sphere in $\su(2)$: 
$$
\Si:=\su(2)\cap SU(2)=\{x \in \su(2): \la x,x\ra=1\}. 
$$
Note that $\Si$ can be regarded as the set of pure unit quaternions; 
i.e.,  we have $x^{2}=-\one$ for all $x \in \Si$. 
We equip $\Si$ with the orientation so that its unit normal vector is $-\one$. 
In other words, we identify $\Si$ with $\Pi_{-\one}$ as oriented hyperplanes. 

 The tangent bundle $TS^{3}$ of $S^{3}=SU(2)$ can be written as 
\begin{align*}
TS^{3}=\{(x, \nu) \in \matH \times \matH : \la x,x\ra=1, \, \la x,\nu\ra=0\}. 
\end{align*}
The {\it left and right Maurer-Cartan forms} $\om$, $\om':TS^{3} \to \su(2)$ 
are defined as follows: 
\begin{align*}
 \om(x,\nu):=x^{-1}\nu,  \quad \om'&:TS^{3} \to \su(2). 
\end{align*}
The restrictions of $\om$ (resp. $\om'$) to 
the unit tangent bundle 
\begin{align*}
T^{1}S^{3}
&=\{(x, \nu) \in \matH \times \matH : \la x,x\ra=\la \nu,\nu\ra=1, \,
\la x,\nu\ra=0\} \\
&=\{(x, \nu) \in SU(2) \times SU(2) : \la x,\nu\ra=0\}
\end{align*}
of $S^{3}  = SU(2)$ is also denoted by $\om$ (resp. $\om'$). 
Note that the maps $\om,\om':T^{1}S^{3} \to \su(2)$ have images in $\Si=\su(2) \cap SU(2)$. 
Observe that since $g^{2}=-\one$ for every $g \in \Si$, we have
\begin{align}
\om(x,\nu)&=x^{-1}\nu=-\nu^{-1}x, \\
\quad \om'(x,\nu)&=\nu x^{-1}=-x\nu^{-1}
\end{align}
for every $(x,\nu) \in T^{1}S^{3}$. 

\begin{rem}
The space $\calG(S^{3})$ of oriented closed geodesics on $S^{3}$ can be identified with 
the quotient space $T^{1}S^{3}/\sim$,  where $\sim$ is an equivalent relation in $T^{1}S^{3}$ 
defined by $(x,v) \sim (x',v')$ if and only if  $x'=(\cos t) x+(\sin t) v$ and 
$v'=-(\sin t) x+(\cos t) v$ for some $t \in \R$. 
It can easily seen that the map $\om \times \om':T^{1}S^{3} \to \Si \times \Si$ induces a bijective map
$\calG(S^{3}) \to \Si \times \Si$. 
This gives a concrete expression of the identification 
$\calG(S^{3}) \cong \Gr_{2}(\R^{4}) \cong S^{2} \times S^{2}$, 
where $\Gr_{2}(\R^{4})$ is the Grassmannian manifold of 
oriented $2$-dimensional subspaces in $\R^{4}$. 
\end{rem}

\begin{lem}
Given a hyperplane $\Pi_{\nu} \,(\nu \in S^{3})$ of $S^{3}$, 
the maps $\Gamma, \Gamma':\Pi_{\nu} \to \Si$ defined by 
\begin{align*}
\Gamma(x):=\om(x,\nu), \quad \Gamma'(x):=\om'(x,\nu)
\end{align*}
are orientation preserving isometries. 
\end{lem}

\begin{proof}
For an element $g \in S^{3} = SU(2)$, the left translation $L_{g}:S^{3} \to S^{3}$ defined by 
$L_{g}(x)=gx$ is an isometry. 
Since 
$$
\Gamma(x)=\om(x,\nu)=x^{-1}\nu=-\nu^{-1}x=L_{-\nu^{-1}}(x), 
$$ 
the map $\Gamma=L_{-\nu^{-1}}:\Pi_{\nu} \to \Pi_{-\one}=\Si$ is an orientation preserving isometry. 
The proof for $\Gamma'$ is almost the same. 
\end{proof}

\begin{rem}
The map $\Gamma:\Pi_{\nu} \to \Si$ can be expressed as the composition $\om \circ \iota$, 
where $\iota:\Pi_{\nu} \to T^{1}S^{3}; \, x \to (x,\nu)$ is the natural immersion 
that assigns each $x \in \Pi_{\nu}$ to the unit normal vector $(x,\nu) \in T^{1}_{x} S^{3}$ to $\Pi_{\nu}$.  
Therefore, the map $\Gamma$ can be regarded as a kind of Gauss map of $\Pi_{\nu}$. 
The same can be said for $\Gamma'$.  
\end{rem}

\section{Proof of the Main Theorem}

In this section, we give the proof of the main theorem. 
Let $P$ be a convex polyhedral sphere in $S^{3}$, 
and let $\{f_{i}\}_{i=1}^{l}$, $\{e_{j}\}_{j=1}^{m}$ and 
$\{v_{k}\}_{k=1}^{n}$ be the sets of faces, edges and vertices of $P$, respectively. 
Recall from Section 5 that the set $\wt P$ of outward unit normal vectors to $P$ is tiled by polygons in 
$\Poly(\tilde P)=\{\tilde f_{i}\}_{i=1}^{l} \cup \{\tilde e_{j}\}_{j=1}^{m} \cup \{\tilde v_{k}\}_{k=1}^{n}$. 
In what follows, the restriction $\om|_{\wt P}$ of 
the left Maurer-Cartan form $\om:T^{1}S^{3} \to \Si$ to $\wt P$ is also denoted by $\om$.  
Similarly, $\om'|_{\wt P}$ is also denoted by $\om'$. 
\begin{prop}
The map $\om:\wt P \to \Si$ has the following properties: 
\begin{enumerate}
\item $\om:\wt P \to \Si$ is continuous. 
\item For every $1 \le i \le l$, the map 
$\om|_{\tilde f_{i}}:\tilde f_{i} \to \Si$ is an orientation preserving isometry onto its image. 
\item For every $1 \le k \le n$, the map $\om|_{\tilde v_{k}}:\tilde v_k \to \Si$ is an orientation preserving isometry
onto its image. 
\item For every $1 \le j \le m$, the image $\om(\tilde e_{j})$ of $\tilde e_{j}$ 
is a geodesic segment on $\Si$. 
\end{enumerate}
The map $\om':\wt P \to \Si$ also has the same properties.  
\end{prop}

\begin{proof}
The proof of (1) is obvious.  The proof of (2) follows from Lemma 6.2. 
The proof of (3)  is also easy. 
We now prove (4). 
Let $K$ be a one-parameter subgroup of $SU(2)$ defined by 
$$
K=\left\{\exp(t \mi)=\left(\begin{array}{cc}
e^{it}  &0   \\
 0 & e^{-it} 
\end{array}\right) : t \in \R\right\}. 
$$
Set $x_{t}=\exp( t \mi)$. 
Since $x_{t}=(\cos t)\one+(\sin t)\mi$, it can be seen that $K=S^{3} \cap \Span_{\R}(\one,\mi)$, 
which implies that  $K$ is a geodesic on $S^{3}=SU(2)$. 
Its dual geodesic is written as 
$S^{3} \cap \Span_{\R}(\mj,\mk)=K \mj=\{y_{s}:=x_{s} \mj: x_{s} \in K\}$. 

By applying an isomety to $S^{3}$ if necessary,  
we may assume that $e_{j} \subset K$ and $e_{j}^{*} \subset K \mj$. 
Furthermore, we may assume that $e_{j}$ and $e_{j}^{*}$ are written in the forms  
$e_{j}=\{x_{s}:  0 \le s \le a\}$ and $e_{j}^{*}=\{y_{t}:0 \le t \le b\}$, respectively. 
Then for $(x_{s},y_{t}) \in \tilde e_{j}$, we have 
$\om(x_{s},y_{t})=x_{s}^{-1}y_{t}=x_{-s}x_{t}\mj=y_{t-s}$, 
which implies that the image $\om(\tilde e_{j})=\{y_{u}: -a \le u \le b-a\}$ of $\tilde e_{j}$
 lies in the  geodesic $K \mj$ on $S^{3}$. 
 
 The proof of the statement for $\om'$ is almost the same; 
 however, in this case, we have $\om'(x_{s},y_{t})=y_{t}x_{s}^{-1}=y_{t+s}$ for $(x_{s},y_{t}) \in \tilde e_{j}$.  
\end{proof}

We set subsets of $\Si$ as 
\begin{align*}
\begin{array}{lll}
F_{i}:=\om(\tilde f_{i}), \quad & E_{j}:=\om(\tilde e_{j}), \quad & V_{k}:=\om(\tilde v_{k}), \\
F_{i}':=\om'(\tilde f_{i}), \quad & E_{j}':=\om'(\tilde e_{j}), \quad & V_{k}':=\om'(\tilde v_{k})
\end{array}
\end{align*}
for $1 \le i \le l$, $1 \le j \le m$ and $1 \le k \le n$. 
Then $F_{i} ,\, F_{i}' \, (i=1, \ldots, l)$ and $V_{k},\, V_{k}' \, (k=1, \ldots, n)$ 
are spherical polygons in $\Si$ with geodesic boundaries and 
$E_{j}, \,E_{j}'\, (j=1, \ldots, m)$ are geodesic segments on $\Si$. 
The following theorem is the precise statement of the main theorem: 
\begin{thm} The spherical polygons 
${F}_{1}, \ldots ,{F}_{l}, {V}_{1}, \ldots ,{V}_{n}$  provide a non-edge-to-edge tiling of $\Si$. 
More precisely, we have 
\begin{enumerate}
\item The polygons ${F}_{1}, \ldots ,{F}_{l}, {V}_{1}, \ldots ,{V}_{n}$  cover all of $\Si$; i.e., 
\begin{align*}
\Si=\left( \bigcup_{i=1}^{l}F_{i}\right) \cup \left( \bigcup_{k=1}^{n}V_{k}\right). 
\end{align*}
\item 
The interiors 
$\mathring{F}_{1}, \ldots ,\mathring{F}_{l}, \mathring{V}_{1}, \ldots ,\mathring{V}_{n}$ 
of polygons ${F}_{1}, \ldots ,{F}_{l}, {V}_{1}, \ldots ,{V}_{n}$ 
are mutually disjoint. 
\end{enumerate}
Similarly, the spherical polygons 
${F}_{1}', \ldots ,{F}_{l}', {V}_{1}', \ldots ,{V}_{n}'$  also provide a non-edge-to-edge tiling of $\Si$. 
\end{thm}

In the proof of Theorem 7.2, we make use of the following:
\begin{lem}
If $f:S^{2} \to S^{2}$ is a local homeomorphism, then $f$ is a homeomorphism. 
\end{lem}
This is a special case of a famous theorem whose statement is as follows: 
Suppose $X$ and $Y$ are topological manifolds and $f:X \to Y$ is a proper local homeomorphism. 
Then $f$ is a covering map (see for example \cite{For}, Theorem 4.22). 

\begin{proof}[Proof of Theorem 7.2]
We are mainly concerned with the statement about the polygons 
$\{F_{i}\}_{i=1}^{l} \cup \{V_{k}\}_{k=1}^{n}$ 
because the arguments for the polygons $\{F_{i}'\}_{i=1}^{l} \cup \{V_{k}'\}_{k=1}^{n}$ are almost the same. 
We first observe the local configuration of these polygons.  
Take an edge $e_{j}$ of $P$ and choose an orientation of $e_{j}$ arbitrarily.   
Let $f_{\alp_{1}}, f_{\alp_{2}}$ be the faces of $P$ that meet along edge $e_{j}$, 
and suppose that $f_{\alp_{1}}$ is on the left side of $e_{j}$ 
and $f_{\beta_{2}}$ is on the right side.  
Let  $v_{\beta_{1}}$  (resp.  $v_{\beta_{2}}$) be  the starting (resp. terminal) point of $e_{j}$. 
As in the proof of Proposition 7.1 (4), we may assume that 
$e_{j}$ and $e_{j}^{*}$ are in the forms 
$e_{j}=\{x_{s}:  0 \le s \le a\}$ and $e_{j}^{*}=\{y_{t}:0 \le t \le b\}$, respectively. 
In addition, we may assume that the orientations on $e_{j}$ and $e_{j}^{*}$ 
coincide with those induced from the parametrizations given by $x_{s}$ and $y_{t}$, respectively. 
Then, $E_{j}=\om(\tilde e_{j})$ is of the form 
$E_{j}=\{y_{u}: -a \le u \le b-a\} \subset K \mj$.  
The proof of Proposition 7.1 (4) shows that $F_{\alp_{1}}$ is on the right side of $E_{j}$ 
and $F_{\alp_{1}} \cap E_{j}=\{y_{u}: -a \le u \le 0\}$, where 
$E_{j}$ is oriented from the parametrization of $y_{t}$. 
Similarly, $F_{\alp_{2}}$ is on the left side of $E_{j}$  
and $F_{\alp_{2}} \cap E_{j}=\{y_{u}: b-a \le u \le b\}$. 
In addition, $V_{\beta_{1}}$ (resp. $V_{\beta_{2}}$) is on the right side (resp. left side) of $E_{j}$ 
and $V_{\beta_{1}} \cap E_{j}=\{y_{u}: 0 \le u \le b\}$ 
(resp.$V_{\beta_{2}} \cap E_{j}=\{y_{u}: -a \le u \le b-a\}$). 
Furthermore, observe that 
$F_{\alp_{1}}$ and $V_{\beta_{1}}$, as well as $F_{\alp_{2}}$ and $V_{\beta_{2}}$, 
meet along a geodesic segment,  
because $E_{k}=\om(\tilde e_{k})$ is a geodesic segment for each $1 \le k \le n$. 

\begin{figure}[h]
\begin{center}
\includegraphics[height=6cm]{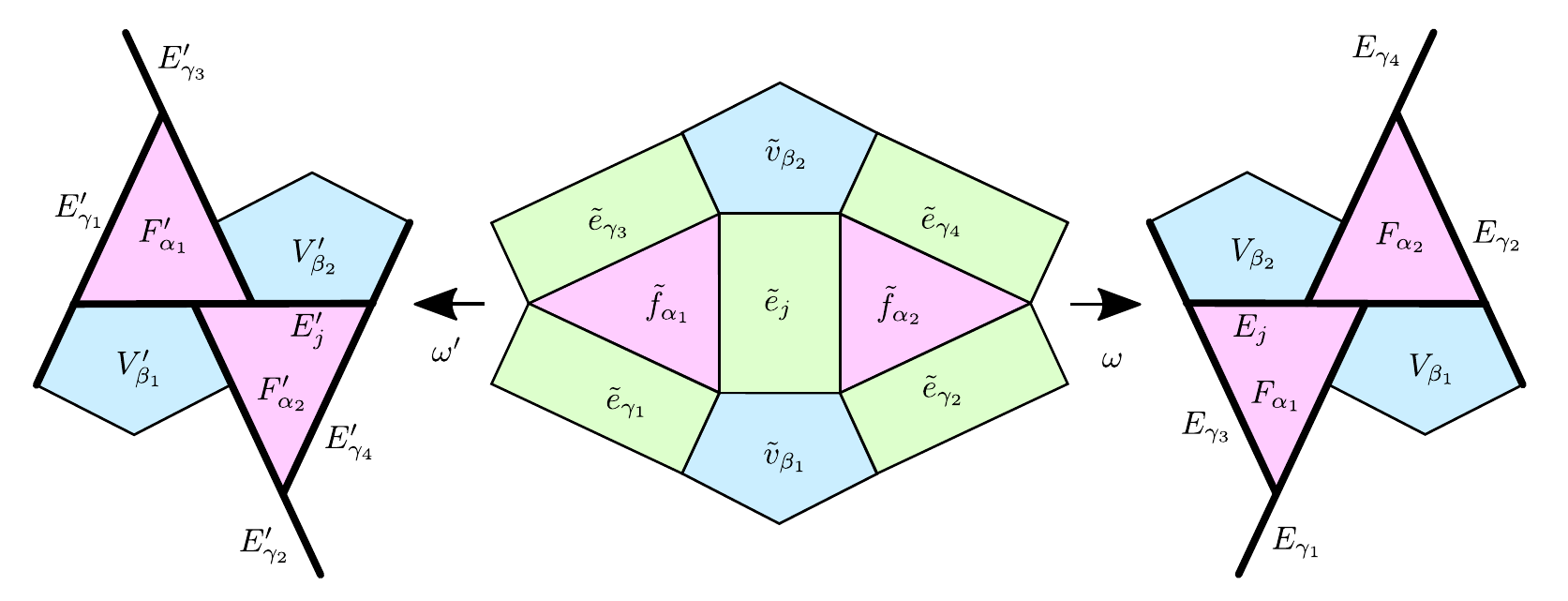}
\caption{Polygons in $\wt P$ adjacent to $\tilde e_{j}$ and 
their images in $\Si$ under the maps $\om$ and $\om'$. }
\end{center}
\end{figure}

So far, we have observed that polygons in $\wt P$ adjacent to $\tilde e_{j}$ 
are mapped by $\om$ to polygons and geodesic segments in $\Si$ 
to give local non-edge-to-edge tilings. 
What we need to show is that $\{F_{i}\}_{i=1}^{l}$ and $\{V_{k}\}_{k=1}^{n}$ give 
non-edge-to-edge tiling of whole $\Si$. 
To this end, for every sufficiently small $\ep>0$, we construct local homeomorphism
$$
\om^{\ep}:\wt P \to \Si, 
$$
such that $\om^{\ep}$ converges uniformly to $\om$ as $\ep$ tends to $0$. 
For each $i =1, \ldots, l$, we construct a geodesic polygon $F_{i}^{\ep}$ in $\Si$,  
which lies in $F_{i}$ as follows: 
Suppose that $F_{i}$ is an $a_{i}$-gon, and 
let $v_{i,1}, \ldots, v_{i,a_{i}}$ be vertices of $F_{i}$, which are 
negatively ordered in $\bd F_{i}$.  In addition, we set $v_{i,a_{i}+1}=v_{i,1}$. 
For each $j=1, \ldots, a_{i}$, 
let $v_{i,j}^{\ep}$ be the point on the edge of $F_{i}$ joining $v_{i,j}$ and $v_{i,j+1}$  
such that $d(v_{i,j},v_{i,j}^{\ep})=\ep$ (see Figure 5 (left)). 
We then define a polygon $F_{i}^{\ep}$ to be the geodesic polygon with 
vertices $v_{i,1}^{\ep}, \ldots, v_{i,a_{i}}^{\ep}$, 
which are negatively ordered in $\bd F_{i}^{\ep}$. 

For each $k=1, \ldots, n$, we similarly construct a geodesic polygon 
$ V_{k}^{\ep}$ in $V_{k}$ as follows: 
Suppose that $V_{k}$ is an $b_{k}$-gon, and 
let $w_{k,1}, \ldots, w_{k,b_{k}}$ be vertices of $V_{k}$, which are 
{\it positively} ordered in $\bd V_{k}$.  In addition, we set $w_{k,b_{k}+1}=w_{k,1}$. 
For each $j=1, \ldots, b_{k}$, 
let $w_{k,j}^{\ep}$ be the point on the edge of $V_{k}$ joining $w_{k,j}$ and $v_{k,j+1}$ 
such that $d(w_{k,j},w_{k,j}^{\ep})=\ep$ (see Figure 5 (right)). 
We then define a polygon $V_{k}^{\ep}$ to be the geodesic polygon with 
vertices $w_{k,1}^{\ep}, \ldots, w_{k,b_{k}}^{\ep}$, which are positively ordered in $\bd V_{k}^{\ep}$. 

\begin{figure}[h]
\begin{center}
\includegraphics[height=5cm]{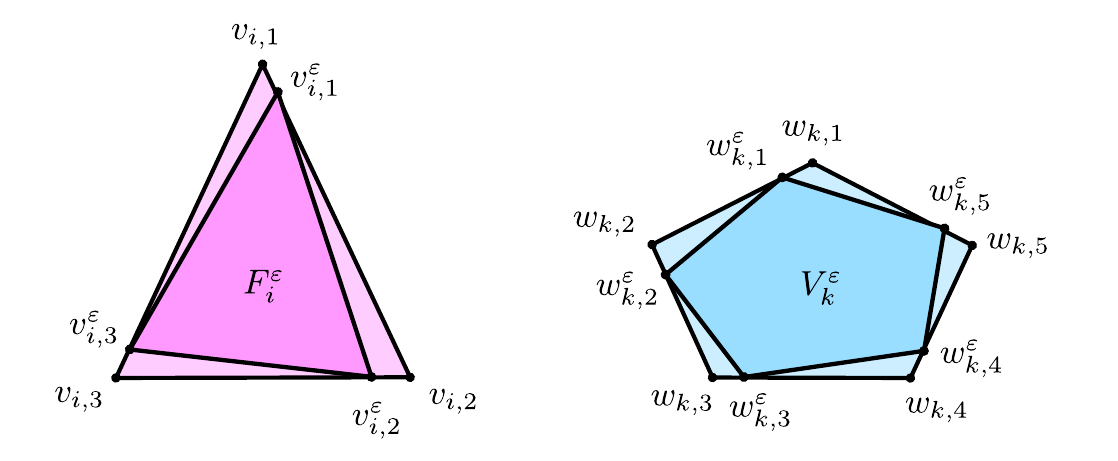}
\caption{Polygons $F_{i}^{\ep} \subset F_{i}$ (left) and $V_{k}^{\ep} \subset V_{k}$ (right) 
in the case where $F_{i}$ is a triangle 
and $V_{k}$ is a pentagon. }
\end{center}
\end{figure}

To define the map $\om^{\ep}:\wt P \to \Si$,  we define $\om^{\ep}$ on each polygon in $\Poly(\wt P)$. 
In what follows, for two polygons $\wp,\wp'$, we say that a homeomorphism $f:\wp \to \wp'$ 
{\it polygonal} if $\wp$ and $\wp'$ have the same number of vertices,  
and $f$ takes vertices of $\wp$ to those of $\wp'$. 
For every sufficiently small $\ep>0$, 
we construct orientation preserving polygonal 
homeomorphisms 
$$
\om^{\ep}|_{\tilde f_{i}}:\tilde f_{i} \to F_{i}^{\ep}
$$
such that  $\om^{\ep}|_{\tilde f_{i}}$ converges uniformly to $\om|_{\tilde f_{i}}$ as $\ep \to 0$. 
Similarly, we construct orientation preserving polygonal  homeomorphisms 
$\om^{\ep}|_{\tilde v_{k}}:\tilde v_{k} \to V_{k}^{\ep}$. 

Given an edge $e_{j}$ of $P$,  take faces $f_{\alp_{1}}, f_{\alp_{2}}$ and vertices 
$v_{\beta_{1}}, v_{\beta_{2}}$ of $P$ as in the begining of this proof. 
Recall that polygons $F_{\alp_{1}}, F_{\alp_{2}}, V_{\beta_{1}}$ and $V_{\beta_{2}}$ in $\Si$ 
meet along geodesic segment $E_{j}$ (Figure 4 (right)). 
Then one see that the polygons 
$F_{\alp_{1}}^{\ep}, F_{\alp_{2}}^{\ep}, V_{\beta_{1}}^{\ep}$ and $V_{\beta_{2}}^{\ep}$ surround a parallelogram, which is denoted by $E_{j}^{\ep}$ (see Figure 6).   
We then define polygonal homeomorphisms 
$\om^{\ep}|_{\tilde e_{j}}:\tilde e_{j} \to E_{j}^{\ep}$ so that 
$\om^{\ep}|_{\tilde e_{j}}$ coincides with 
$\om^{\ep}|_{\tilde f_{\alp_{1}}}$ on $\tilde e_{j} \cap \tilde f_{\alp_{1}}$, and so on. 
We also asuume that $\om^{\ep}|_{\tilde e_{j}}$ converge uniformly to 
$\om|_{\tilde e_{j}}$ as $\ep \to 0$.  
\begin{figure}[h]
\begin{center}
\includegraphics[height=5.5cm]{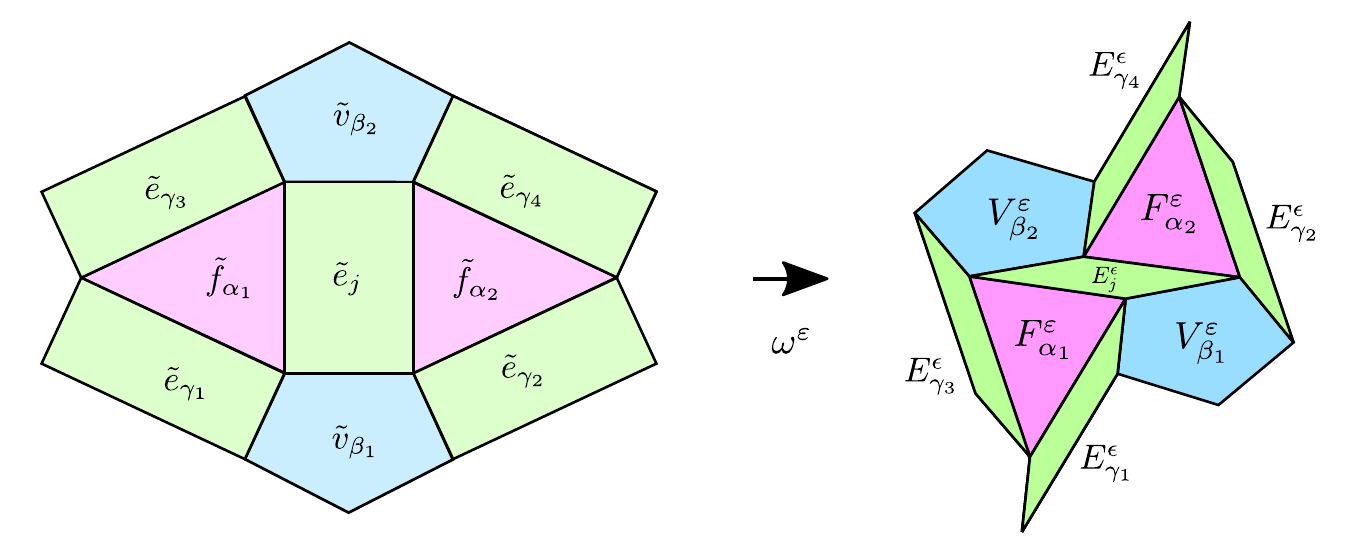}
\caption{Polygons in $\wt P$ adjacent to $\tilde e_{j}$ and 
their images in $\Si$ under the map $\om^{\ep}$.
 Compare this figure with Figure 4.}
\end{center}
\end{figure}
Then, the polygonal homeomorphisms $\om^{\ep}|_{\tilde f_{i}} \, (1 \le i \le l)$, 
$\om^{\ep}|_{\tilde v_{k}}\, (1 \le k \le n)$ and $\om^{\ep}|_{\tilde e_{j}} \, (1 \le j \le m)$  
constructed above can be combined to obtain a local homeomorphism $\om^{\ep}:\wt P \to \Si$. 
It then follows from Lemma 7.3 that $\om^{\ep}:\wt P \to \Si$ is a homeomorphism. 
Therefore, 
$\mathring{F}_{1}^{\ep}, \ldots ,\mathring{F}_{l}^{\ep}, \mathring{V}_{1}^{\ep}, \ldots ,\mathring{V}_{n}^{\ep}$
are mutually disjoint, and 
\begin{align*}
\Si=\left( \bigcup_{i=1}^{l}F_{i}^{\ep}\right) 
\cup 
\left( \bigcup_{j=1}^{m}E_{j}^{\ep}\right)
\cup 
\left( \bigcup_{k=1}^{n}V_{k}^{\ep}\right). 
\end{align*}
Since $\om^{\ep}$ converges to $\om$ uniformly as $\ep \to 0$, 
the sets $F_{i}^{\ep}$, $E_{j}^{\ep}$ and $V_{k}^{\ep}$ converge to $F_{i}$, 
$E_{j}$ and $V_{k}$, respectively, in the sence of Hausdorff distance.  
This completes the proof of Theorem 7.2. 
\end{proof}

\end{document}